\documentclass[12pt,reqno]{amsart}

\usepackage[x11names]{xcolor}
\usepackage[linkcolor=Blue2,citecolor=red,colorlinks]{hyperref}
\usepackage[capitalize,nameinlink]{cleveref}
\usepackage{color, bm, amscd, tikz-cd}
\newcommand{\cA}{{\mathcal A}}
\newcommand{\cB}{{\mathcal B}}
\newcommand{\cC}{{\mathcal C}}
\newcommand{\cD}{{\mathcal D}}
\newcommand{\cE}{{\mathcal E}}

\newcommand{\cH}{{\mathcal H}}

\newcommand{\cK}{{\mathcal K}}

\newcommand{\cM}{{\mathcal M}}
\newcommand{\cN}{{\mathcal N}}

\newcommand{\cS}{{\mathcal S}}

\newcommand{\bT}{{\mathbb{T}}}

\newcommand{\bG}{{\mathbb G}}
\newcommand{\bC}{{\mathbb C}}

\newenvironment{psmallmatrix}
{\left(\begin{smallmatrix}}
	{\end{smallmatrix}\right)}

\newtheorem{thm}{Theorem}[section]

\newtheorem{corollary}[thm]{Corollary}
\newtheorem{lemma}[thm]{Lemma}

\theoremstyle{definition}

\newtheorem{definition}[thm]{Definition}
\newtheorem{remark}[thm]{Remark}

\newcommand\blfootnote[1]{%
	\begingroup
	\renewcommand\thefootnote{}\footnote{#1}%
	\addtocounter{footnote}{-1}%
	\endgroup
}

\numberwithin{equation}{section}

\author[Bhowmik and Kumar]{Mainak Bhowmik and Poornendu Kumar}

\address{Department of Mathematics\\
	Indian Institute of Science\\
	Bangalore, 560012, India} 
	\email{mainakb@iisc.ac.in; mainak.bhowmik943@gmail.com}
	
\address{Department of Mathematics\\
	University of Manitoba\\
Winnipeg, Manitoba, R3T 2N2, Canada.}
\email{poornendu.kumar@umanitoba.ca}

\begin{document}
	\title[Realization of operator-valued Herglotz function]{Herglotz representation for operator-valued function on a set associated with test functions}
\maketitle

	{\blfootnote{{2020 {\em Mathematics Subject Classification}: 47A48, 47A56, 32A38.  \\
			{\em Key words and phrase}:  Herglotz representation, Realization formulae, Krein spaces, Test functions, Kolmogorov decomposition, Herglotz-Agler class.}}}

\begin{abstract}
	
	The Herglotz representation theorem for holomorphic functions with non-negative real part is a fundamental result in the theory of holomorphic functions. In this paper, we reinterpret the Herglotz representation in the context of modern techniques, specifically realization formula. This reinterpretation is then extended to operator-valued functions on arbitrary sets, in association with a collection of test functions.

\end{abstract}
\section{Introduction}

Given a regular Borel measure $\mu$ on the unit circle $\mathbb{T}$, the function
\begin{equation}
	f(z)= i \textup{Im}f(0) + \int_{\mathbb T} \frac{\alpha+z}{\alpha-z}d\mu(\alpha) \label{Her}\\
\end{equation}
is holomorphic on the open unit disc $\mathbb{D}$ with a non-negative real part. Interestingly, any holomorphic function on $\mathbb{D}$ having non-negative real part can be uniquely expressed through a positive regular Borel measure on $\mathbb{T}$ as depicted in equation \eqref{Her}. This fundamental result is attributed to Herglotz \cite{Her}. Moreover, if we assume the normalizing condition $f(0)=1$, the measure $\mu$ corresponds to a regular Borel probability measure on $\mathbb{T}$. 

\begin{definition} 
	Given a domain $\Omega$ in $\mathbb C^d$ and a Hilbert space $\mathcal E$, the Herglotz class $\mathfrak H(\Omega, \mathcal B (\mathcal E))$ consists of all $\mathcal B(\mathcal E)$-valued holomorphic functions on $\Omega$ which have non-negative real parts.
\end{definition}
  It is natural to inquire whether we can generalize the Herglotz representation for operator-valued functions on arbitrary domains. Achieving such a integral represenation on any general domain is typically challenging. The main difficulty is to have suitable integrand for the representation.
 Nevertheless, a version of Herglotz's integral representation is known for multi-connected domains \cite{AHR, BH, KH}, the polydisc \cite{Kor-Puk}, the Euclidean unit ball \cite{Kor-Puk, Jury, MP}, and domains possessing certain convexity properties \cite{AZ}  where the integrand is concocted with the reproducing kernel of the Hardy Hilbert spaces of these domains. In this note, the Herglotz representation is extended  for a  class of  functions on an arbitrary set by employing the theory of operators on Hilbert spaces, subject to certain conditions on the set.
 
 The power of Herglotz's theorem is well-known - see for example \cite{DM} where the Herglotz representation has been crucially used in showing failure of rational dilation on a triply connected planar domain
 or \cite{AHR} for a proof of the von Neumann inequality using it. Herglotz’s theorem has applications to de Branges-Rovnyak spaces (see the seminal work \cite{Sar}). Also, it has applications to two-phase composite materials, where the scalar-valued Herglotz functions correspond to the effective properties of the composite materials and the operator-valued Herglotz functions are applied to study the permeability tensor of a porous material, see \cite{OL}. Recently, the Carath\'eodory approximation result was proven in \cite{BBK} using Herglotz's theorem.

 Note that the equation \eqref{Her} connects with the theory of operators on a Hilbert space in an obvious way. By considering $L^2(\mathbb{T}, \mu)$ and the normal operator $U$ (with spectrum contained in the unit circle and hence a unitary operator in this case) of multiplication by the co-ordinate $z$, the equation above takes the form
\begin{equation}
	f(z) = i \textup{Im}f(0) + (\langle (U + z)(U - z)^{-1} 1 , 1 \rangle) \textup{Re}f(0) \label{RF}
\end{equation}
where $1$ is the constant function $1$.  Conversely, for any unitary operator $U$ on some Hilbert space $\cM$ and a unit vector $\xi$ the function
\begin{equation*}
	f(z) = i \textup{Im}f(0) + \langle (U + z)(U - z)^{-1} \xi , \xi\rangle_{\cM} \textup{Re}f(0) 
\end{equation*}
is in the Herglotz class. We shall call \eqref{RF} a {\em realization formula} for the Herglotz class functions.

The realization formula \eqref{RF} has been extended to the bidisc by Agler, Tully-Doyle, and Young \cite{Agler, AY-Book}. Additionally, Ball and Kaliuzhnyi-Verbovetskyi \cite{Ball-V} have generalized it for a subclass of $\mathfrak H(\mathbb{D}^d)$, including extensions to operator-valued functions. Moreover, it has been explored in non-commutative setups \cite{PPD}. In this note, by leveraging the realization formula on a domain associated with the test functions which we will discuss this in more detail in Section \ref{Re}, along with Krein space theory and the Kolmogorov decomposition, we extend the Herglotz representation \eqref{RF} to arbitrary domains associated with the test functions, which is a more modern theme now. The main theorem of this paper is the following:  
\begin{thm}\label{Main}
	Let $\Psi$ be a collection of test functions on a set $X$ such that there exists a point $x_0\in X$ with $\psi(x_0)=0$ for all $\psi\in \Psi$. Let $\mathcal{E}$ be a Hilbert space and $f:X\rightarrow \mathcal{B}(\mathcal{E})$ be a function. Then the following are equivalent:
	
	\begin{itemize}
		\item[\textbf{[H]}]
		$f$ is a $\Psi$-Herglotz-Agler class function;
		
		\item [\textbf{[D]}]
		there exists a completly positive kernel $\Lambda: X\times X\rightarrow \mathcal{B}\left( C_b(\Psi), \mathcal{B}(\cE)\right)$ such that
		\begin{align}\label{EqnD2}
			f(x)+ f(y)^* = \Lambda (x,y) \left(1-  E(x)\overline{E(y)} \right)
		\end{align}
		
		\item [\textbf{[R]}] 
		there exist a Hilbert space $\mathcal{K}$, a unital $*$-representation $\pi: C_b(\Psi)\rightarrow\mathcal{B}(\mathcal{K})$ and a bounded colligation operator matrix
		$$U= \begin{bmatrix}
			A & B \\
			C & D
		\end{bmatrix}: \mathcal{E}\oplus\mathcal{K}\rightarrow \mathcal{E}\oplus\mathcal{K}$$
		such that 	\begin{align} \label{realization-general}
			f(x) = A + B\pi\left(E(x)\right) \left(I- D\pi\left(E(x)\right)\right)^{-1}C
		\end{align}
		where $D^*D= DD^*= I_\mathcal{K}$, $D^*C= B^*$ and $C^*C=BB^*= A+A^*$.
	\end{itemize}
\end{thm}
\noindent All the necessary terms in the aforementioned theorem will be discussed in the following section. Additionally, some particular cases of  \cref{Main} as well as examples will be provided in the last section.

Finally, we extend our heartfelt thanks to Prof. Joe Ball and  Prof. Tirthankar Bhattacharyya for valuable discussions and insights.

\section{The Realization Formula and The Kolmogorov Decomposition} \label{Re}

A remarkable result in function-theoretic operator theory is the \emph{realization formula}, which asserts that any holomorphic function \(\varphi: \mathbb{D} \rightarrow \overline{\mathbb{D}}\) can be represented as

$$
\varphi(z)=A +zB(I-zD)^{-1}C 
$$
for an isometry $ U= \left[\begin{smallmatrix}
	A & B \\
	C & D 
\end{smallmatrix}\right]$ on $\mathbb{C}\oplus\cH$, where $\cH$ is a Hilbert space. A holomorphic function $\theta$ on a domain $\Omega \subseteq \mathbb{C}^d$, taking values in $\cB(\cE)$, the $C^*$-algebra of bounded linear operators on some Hilbert space $\cE$, is called a {\em Schur class function} if $\|\theta(z) \| \leq 1$ for all $z$ in $\Omega$. The realization formula for Schur class functions has been extended to various domains such as an annulus \cite{DM}, the bidisc \cite{AM-Book}, the complex unit ball \cite{Ball-Bolo}, and the symmetrized bidisc \cite{SYM_Real, Tirtha-Hari-JFA}. Notably, not every Schur class function on the polydisc $\mathbb{D}^d$ ($d\geq 3$) possesses a realization formula. However, a specific subclass, known as the {\em Schur-Agler class}, of the Schur class on $\mathbb{D}^d$ does; see \cite{Agler}. This concept has further been generalized to an abstract setting where the domain $\Omega$ is replaced by a set $X$, and the Schur class is replaced by a particular class of functions depending on a collection of test functions $\Psi$ defined on $X$. This class is referred to as the {\em $\Psi$-Schur-Agler class}.  Additionally, explore recent interesting work on this class of functions in the operator-valued setting \cite{BBD,  BPT, CHL, Hartz, K2, K3, KRT}.

\noindent A collection $\Psi$ of $\bC$-valued functions on a set $X$ is called a set of {\em
	test functions} (see \cite{DM}) if the following conditions hold:
\begin{enumerate}
	\item $\sup_{\psi\in \Psi} |\psi(x)|<1$ for all $x\in X$;
	\item for each finite subset $F$ of $X$, the collection $\{\psi|_{F}: \psi\in \Psi\}$ together with the constant functions generates the algebra of all $\bC$-valued functions on $F$.
\end{enumerate}
The collection $\Psi$ inherits a subspace topology of the space of all bounded functions from $X$ to $\overline{\mathbb D}$ endowed with topology of pointwise convergence. We shall denote the algebra of bounded continuous functions over $\Psi$ with pointwise algebra operation by $C_b(\Psi)$. Define an injective mapping $E:X\rightarrow C_b(\Psi)$ as $E(x)= e_x$, where $e_x(\psi)=\psi(x).$

\begin{definition}
Let $\cE$ be a Hilbert space. A function $k: X\times X \rightarrow \mathcal{B}(\cE)$ is said to be a $\mathcal{B}(\cE)$-valued positive kernel if for any $n \geq 1$, any $n$ points $x_1, \dots, x_n$ in $X$ and any $n$ vectors $e_1, \dots, e_n$ in $\cE$ we have
$$ \sum_{i,j=1}^n \langle k(x_i, x_j)e_j, e_i \rangle \geq 0 .$$ 
We define a map $\Lambda: X \times X \to \mathcal{B}(C_b(\Psi), \mathcal{B}(\cE))$ to be a completely positive kernel if
\begin{align*}
	\sum_{i,j=1}^n T_i^* \Lambda(x_i, x_j)\left(\overline{f_i}f_j \right)T_j\geq0
\end{align*}for any $n\geq 1$, any $n$ points $x_1, \dots, x_n\in X $, any $n$ operators $T_1, \dots, T_n $ in $\cB(\cE)$ and any $n$ functions $g_1, \dots, g_n$ in $C_b(\Psi)$.
\end{definition}
\noindent Now, we shall state the famous Kolmogorov decomposition of a completely positive kernel which is one of the key ingredients for the proof of our main result.
\begin{thm}[Kolmogorov Decomposition  \cite{BH}]\label{KD}
	Let $\cE$ be a Hilbert space. Then, $\Lambda: X\times X\rightarrow\mathcal{B}\left(C_b(\Psi), \cB(\cE)\right)$ is a completely positive kernel if and only if there exist an auxiliary Hilbert space $\cK$, a map $h: X \rightarrow \cB(\cK, \cE)$ and a unital $*$-representation $\rho: C_b(\Psi)\rightarrow \cB(\cK)$ such that
	$$\Lambda(x, y)(g)= h(x)  \rho(g) h(y)^* $$
	for all $x, y\in X$ and $g\in C_b(\Psi)$. 	
\end{thm}
\noindent Let $\mathcal{K}_{\Psi}(\cE)$ be the collection of $\cB(\cE)$-valued positive kernels $k$ on $X$ such that  the function
$$(x, y)\mapsto\left(1-\psi(x)\overline{\psi(y)}\right)k(x,y)$$
is a $\cB(\cE)$-valued positive kernel, for each $\psi\in\Psi$. We say that $f: X \rightarrow \cB(\cE)$ is in $H^\infty_{\Psi}(\cE)$ if there is a non-negative constant $C$ such that the function
\begin{align}\label{HH}
	(x, y)\mapsto\left(C^2 k(x,y) - f(x)k(x,y) f(y)^* \right)
\end{align}
is a $\cB(\cE)$-valued positive kernel for each $k$ in $\mathcal{K}_{\Psi}(\cE)$. If $f$ is in $H^\infty_{\Psi}(\cE)$, then we denote by $C_f$ the smallest $C$ which satisfies \eqref{HH}. The collection of maps $f\in H^\infty_{\Psi}(\cE)$ for which
$C_f$ is no larger than $1$ is called the {\em $\Psi$-Schur-Agler class} and it is denoted by $\cS\cA_{\Psi}(\cE).$ We shall use the following realization formula for operator-valued $\Psi$-Schur-Agler class functions \cite{Tirtha-Ani-Vik, DM}.

\begin{thm}\label{Thm}
	Let $\cE$ be a Hilbert space and $f: X \to \cB(\cE)$ be a function. Then the following are equivalent:
	
	\begin{enumerate}
		\item $f \in \cS\cA_{\Psi}(\cE)$;
		
		\item there exists a completely positive kernel $\Lambda: X\times X \rightarrow \cB\left(C_b(\Psi), \cB(\cE) \right)$ such that
		$$
		I- f(x) f(y)^* = \Lambda (x, y)\left(1- E(x)\overline{E(y)}\right);
		$$
		\item there exist a Hilbert space $\mathcal{K}$, a unital $*$-representation $\rho: C_b(\Psi) \to \cB(\cK)$ and a unitary operator
		\begin{align*}
			U=\begin{bmatrix}
				A & B\\
				C & D
			\end{bmatrix}
			:\begin{bmatrix}\cE\\\mathcal{K}\end{bmatrix}\to \begin{bmatrix}\cE\\\mathcal{K}\end{bmatrix}
		\end{align*}
		such that
		\begin{align*}
			f(x)= A+B \rho(E(x))(I-D \rho(E(x)))^{-1}C \ \text{for all}\ x\in X.
		\end{align*}
	\end{enumerate}
\end{thm}
It is time to define the key protagonist of this paper.
\begin{definition}
A function $f: X \to \cB(\cE)$ is said to be $\Psi$-\textit{Herglotz-Agler} class function if 
$$ (x,y) \mapsto f(x)k(x, y) + k(x, y) f(y)^*$$
is a $\cB(\cE)$-valued positive kernel for each $k$ in $\mathcal{K}_{\Psi}(\cE)$. 
\end{definition}
\noindent The realization formula is extremely powerful and leads to numerous significant results. For instance, it enables the derivation of the Pick-Nevanlinna interpolation theorem \cite{AM-Book}, proves the commutant lifting theorem \cite{Ball}, and establishes the Carath\' eodory approximation result \cite{ABJK}. It also aids in the factorization of Schur-Aglar class functions \cite{BK, DJ}. Additionally, its applications extend to fields such as signal processing \cite{Gohberg} and electrical engineering \cite{Helton}. In the following section, we will use it to prove  \cref{Main}.

\section{Proof of Main Theorem}
Now, we are all set to prove the the realization of  operator-valued Herglotz functions on a set  $X$ i.e.,  \cref{Main}. The proof is mainly motivated from \cite{Ball-V}. Also, one can explore \cite{Ball-V2} as well as \cite{Ball-IEOT} for insights into the application of Krein space theory in the realization formula.

\noindent \textbf{Proof of Theorem \ref{Main}:}
We establish the equivalence of $\textbf{[H]}, \textbf{[D]}$ and  $\textbf{[R]}$ by showing that
$$ \textbf{[H]}\implies \textbf{[D]} \implies \textbf{[R]}\implies \textbf{[D]}\implies \textbf{[H]}.$$

\noindent $\text {Proof of } \textbf{[H]}\Rightarrow \textbf{[D]}:$	Let	$f$ be in  $\Psi$- Herglotz-Agler class function. Then $$T(x)= \left(I-f(x) \right) \left(I+f(x) \right)^{-1}$$ is in $\Psi$-Schur-Agler class. By Theorem \ref{Thm}, there exists a completely positive kernel $\Gamma: X\times X \rightarrow \cB\left(C_b(\Psi), \cB(\cE) \right)$ such that
$$
I-T(x)T(y)^*= \Gamma(x, y)\left(1- E(x)\overline{E(y)} \right).
$$
After putting the expression of $T$ we get
$$I- \left(I+f(x) \right)^{-1} \left(I-f(x) \right)  \left(I- f(y)^* \right) \left(I+f(y)^* \right)^{-1} =\Gamma(x, y)\left(1- E(x)\overline{E(y)}\right).$$
A simple calculations yeilds the following
\begin{align}\label{EqD}
	2\left(I+ f(x)\right)^{-1} [f(x)+ f(y)^*] \left(I+f(y)^*\right)^{-1}=\Gamma(x, y)\left(1- E(x) \overline{E(y)} \right).
\end{align}
Set $$\Lambda(x, y) (g) =  \left[\frac{I+ f(x)}{\sqrt{2}}\right]\Gamma(x, y)(g)\left[\frac{I+ f(y)^*}{\sqrt{2}}\right] $$ for $g\in C_b(\Psi)$. Then from equation \eqref{EqD}, we have
$$f(x)+ f(y)^* = \Lambda (x,y) \left(1-  E(x)\overline{E(y)} \right).$$ If we show that $\Lambda(x, y)$ is completely positive kernel, then we are done. This is readily observed by applying Kolmogorov decomposition (Theorem \ref{KD}) to the completely positive kernel $\Gamma$. Indeed, if $\Gamma(x,y)(g)= h(x)\rho(g)h(y)^*$ is a Kolmogorov decomposition of $\Gamma$ where $\rho$ is a unital $*$-representation of $C_b(\Psi)$, then from the definition of $\Lambda$ in terms of $\Gamma$, it is easy to read off that 
$$
\Lambda(x,y)(g)= H(x)\rho(g) H(y)^*
$$
where $H(x)= \frac{1}{\sqrt{2}}(I+f(x))h(x)$. Thus $\Lambda$ has Kolmogorov decomposition and hence it is a completely positive kernel.

\noindent$\text{Proof of}\ \textbf{[D]}\Rightarrow \textbf{[R]}:$ Suppose $f$ satisfies \eqref{EqnD2}, for some completely positive kernel $\Lambda$. Then by Theorem \ref{KD}, there exist a Hilbert space $\cK$ and a unital $*$-homomorphism $\rho : C_b(\Psi) \to \cB(\cK)$ and a function $h: X \to \cB(\cK, \cE)$ such that $\Lambda(x,y)(g)=  h(x)\rho(g)h(y)^*$ for $g\in C_b(\Psi)$. Then we have
\begin{align*}
	f(x)+f(y)^* =   h(x) \rho(1-E(x)\overline{E(y)} ) h(y)^*.
\end{align*}
This can be rewritten as
\begin{align}\label{lurking-eq}
	f(x)+f(y)^* +  h(x) \rho(E(x)) \rho(E(y))^*h(y)^* = h(x) h(y)^*.
\end{align}
Now, we shall employ the Krein-space theory, see \cite{Krein-Bog} for more details. Consider the Krein-space $\mathcal{E}\oplus\mathcal{K}\oplus\mathcal{E}\oplus\mathcal{K}$ with inner-product induced by the signature matrix  given by
\begin{align}\label{signature}
	J= \begin{bmatrix}
		0 & 0 & I_{\cE} & 0 \\
		0 & -I_{\cK}& 0 & 0 \\
		I_{\cE} & 0  & 0 & 0\\
		0 &  0 & 0 & I_{\cK}
	\end{bmatrix}.
\end{align}
Consider the subspace
\begin{align*}
	\cN= \overline{\operatorname{span}}\left\lbrace
	\begin{bmatrix}
		f(x)^* \\
		h(x)^*\\
		I_{\mathcal{E}} \\
		\rho(E(x))^* h(x)^*	
	\end{bmatrix}e: e\in\mathcal{E}\ \text{ and } x\in X
	\right\rbrace
\end{align*}
on the space $\mathcal{E}\oplus\mathcal{K}\oplus\mathcal{E}\oplus\mathcal{K}$. Let $x, y \in X$ and $e, e'\in\mathcal{E}$. Then we have
\begin{align*}
	&\left\langle J 	\begin{bmatrix}
		f(x)^* \\
		h(x)^*\\
		I_{\mathcal{E}} \\
		\rho(E(x))^* h(x)^*	
	\end{bmatrix}e, 	\begin{bmatrix}
		f(y)^* \\
		h(y)^* \\
		I_{\mathcal{E}}\\
		\rho(E(y))^*h(y)^*
	\end{bmatrix}e'\right\rangle\\
	&=\left\langle \begin{bmatrix}
		I_{\mathcal{E}} \\ -h(x)^* \\  f(x)^* \\   \rho(E(x))^*h(x)^* \end{bmatrix}e,  \begin{bmatrix}
		f(y)^* \\
		h(y)^*\\
		I_{\mathcal{E}}\\
		\rho(E(y))^*h(y)^*
	\end{bmatrix}e' \right\rangle\\
	& = \left\langle
	e, [f(x)+f(y)^* + h(x) \rho(E(x))\rho(E(y))^*h(y)^*- h(x)h(y)^*] e'
	\right\rangle\\
	& = 0 \quad \quad \quad \quad  \text{ (by equation \eqref{lurking-eq})}.
\end{align*}
Therefore, $\cN$ is a $J$-isotropic subspace on $\mathcal{E}\oplus\mathcal{K}\oplus\mathcal{E}\oplus\mathcal{K}$. Let us proceed with embedding $\cN$ into a $J'$-Lagrangian subspace, denoted as $\cN'$, by considering an ambient space $\mathcal{E}\oplus\mathcal{K}'\oplus\mathcal{E}\oplus\mathcal{K}'$, where $\mathcal{K}\subset \mathcal{K}'$. Notably, the Krein-space Gramian matrix $J'$ in this extended space retains the same block form as $J$ defined in equation \eqref{signature}. Let us demonstrate that $\cN'$ represents a graph of a closed operator with a domain in $\mathcal{E}\oplus\mathcal{K}'$. To establish this, we rely on the following lemma which was communicated to us J. A. Ball. 
\begin{lemma}\label{Graph}
	Let $\mathcal{K}_1$ and $\mathcal{K}_2$ be two Hilbert spaces. Then a closed subspace $\cN \subset \mathcal{K}_1 \times \mathcal{K}_2$ has the form of a graph space, i.e., there exists a closed linear operator (possibly unbounded) $\Delta: \cD \rightarrow \mathcal{K}_1$ defined on some domain $\mathcal{D} \subset \mathcal{K}_2$ so that $X$ has the form
	$$ \cN = \begin{bmatrix}
		\Delta \\
		I_{\mathcal{K}_2}
	\end{bmatrix} \cD$$
	for some linear subspace $\cD$ contained in $\mathcal{K}_2$ if and only if
	$$ X \cap \begin{bmatrix}
		\mathcal{K}_1 \\
		\{0\}
	\end{bmatrix} = \{0\}.$$	
\end{lemma}
We would like to note that a variant of the Lemma \ref{Graph} was the main tool for the development of the Grassmannian approach to Nevanlinna-Pick interpolation developed by Ball and Helton \cite{Ball-Helton-JOT}. Also, see the more recent survey \cite{Ball}. 

\noindent \textbf{Continuation of the proof of} $\textbf{[D]}\Rightarrow \textbf{[R]}:$
To use the Lemma \ref{Graph}, we shall prove that $$\cN'\cap \left(\cE \oplus \cK \oplus \{0\} \oplus \{0\} \right) = \{0\}.$$
To prove this, let $\begin{psmallmatrix}
	e\\ m'\\ 0\\0	
\end{psmallmatrix}$ be a point in $\cN'\cap  \left(\cE \oplus \cK' \oplus \{0\} \oplus \{0\} \right) .$ Then by definition of the set $\cN'$ we have, for all $x\in X$ and $e'\in\mathcal{E}$,
\begin{align*}
	\left\langle J' \begin{bmatrix}
		e\\
		m'\\
		0 \\
		0
	\end{bmatrix}, \begin{bmatrix}
		f(x)^* \\
		h(x)^* \\
		I_{\mathcal{E}}\\
		\rho(E(x))^* h(x)^*
	\end{bmatrix}e' \right\rangle=0.
\end{align*}
A simple calculation yields the following
\begin{align*}
	\langle -h(x)P_{\cK}m'+e, e'\rangle_{\cE}=0
\end{align*}
for all $e'\in\cE$. This implies that $e=h(x) P_{\cK}m'$ for all $x\in X$. In particular, we have $e=h(x_0) P_{\cK}m'$, where $x_0$ is a point in $X$. We also have
\begin{align*}
	0=\left\langle J' \begin{bmatrix}
		e \\
		m'\\
		0 \\
		0
	\end{bmatrix}, \begin{bmatrix}
		e \\
		m'\\
		0 \\
		0
	\end{bmatrix} \right\rangle=\left\langle J' \begin{bmatrix}
		h(x_0) P_{\cK}m' \\
		m'\\
		0 \\
		0
	\end{bmatrix}, \begin{bmatrix}
		h(x_0) P_{\cK}m'\\
		m' \\
		0 \\
		0
	\end{bmatrix} \right\rangle=
	-\|m'\|^2.
\end{align*}
Thus, $m'=0$ and hence $e=0$. This proves that $\cN'\cap\left(\cE \oplus \cK' \oplus \{0\} \oplus \{0\} \right)= \{0\}.$
Therefore by Lemma \ref{Graph}, $\cN'$ is a graph space, i.e., there exists a closed operator $T$ with domain $\mathcal{D} \subset \cE \oplus \mathcal{K}'$ such that
$$
\cN= \left\lbrace  \begin{bmatrix}
	Ty\\
	y
\end{bmatrix} : y\in \mathcal{D} \right\rbrace.
$$
Since for each $x\in X$ and $e\in \cE$, $\begin{psmallmatrix}
	f(x)^*e\\
	h(x)^* e\\
	e\\
	\rho(E(x))^* h(x)^*e
\end{psmallmatrix} \in \cN \subset \cN'$ and $\cN'$ is a graph of $T$, 
$$\begin{pmatrix}
	e\\
	\rho(E(x))^*h(x)^*e	
\end{pmatrix} \in \mathcal{D}.$$
In particular, for $x=x_0$, $E(x_0)$ is the zero function on $\Psi$  and hence $\rho(E(x_0))=0$. So, we have $\begin{psmallmatrix}
	e\\
	0
\end{psmallmatrix} \in \mathcal{D}$. Therefore we can view the subspace $\mathcal{D}$ as $\mathcal{E} \oplus \mathcal{K}''$ for some subspace $\mathcal{K}''$ of $\cK'$. Invoking a variant (considering the signature matrix $J'$) of Lemma 3.4 in \cite{Ball-IEOT}, we get that $\mathcal{K}'' =\cK'$ and the operator $T$ is bounded and it will be of the form $\begin{psmallmatrix}
	A & B\\
	C & D
\end{psmallmatrix}$ on $\mathcal{E} \oplus \cK'$. Also the proof of the same lemma shows that the operators $A, B, C$ and $D$ satisfy the relations in [\textbf{R}].

Consider $x\in X$ and $e\in \cE$. Now $\cN'$ being the graph space of the bounded operator $T$, there exists $$\begin{psmallmatrix}
	e'\\
	m'
\end{psmallmatrix} \in \cE \oplus \cK'$$ such that
\begin{align*} 
	\begin{bmatrix}
		f(x)^*e\\
		h(x)^* e\\
		e\\
		\rho(E(x))^*h(x)^*e
	\end{bmatrix} = \begin{bmatrix}
		A & B \\
		C & D \\
		I_{\cE} & 0 \\
		0  & I_{\cK}
	\end{bmatrix} \begin{bmatrix}
		e'\\
		m'
	\end{bmatrix}.
\end{align*}
This gives,
\begin{align*}
	&Ae' +Bm' = f(x)^*e, \quad 
	Ce' + Dm'= h(x)^*e ,\quad  
	e=e' \quad \text{and} \quad m'= \rho(E(x))^*h(x)^*e.
\end{align*}
Therefore we have \begin{align*}
	&A+B \rho(E(x))^*h(x)^* = f(x)^* \quad \text{and} \quad
	C+ D \rho(E(x))^* h(x)^*  =h(x)^*.
\end{align*}
Note that, $$\|E(x)\|_\infty = \sup_{\psi \in \Psi} |\psi(x)|$$ is strictly smaller than $1$, by our assumption. Also, $D$ being unitary $D\rho(E(x))^*$ is a strict contraction and hence $(I-D\rho(E(x))^*) $ is invertible. Solving the above operator equations we get
$$
f(x)^*= A+B\rho(E(x))^*(\left(I-D\rho(E(x))^*\right)^{-1} C.
$$
Therefore, 
\begin{align*}
f(x) & = A^* + C^*\left( I- \rho(E(x) D^*\right)^{-1} \rho(E(x)) B^* \\
     &= A_1 + B_1 \rho(E(x)) \left( I- D_1  \rho(E(x) \right)^{-1} C_1
\end{align*}
 where $A_1 = A^*$, $B_1= C^*$, $D_1= D^*$ and $C_1= B^*$. It is also easy to check that $U_1 = \left(\begin{smallmatrix}  A_1 & B_1 \\ C_1 & D_1 \end{smallmatrix} \right)$ is unitary and the operators $A_1, B_1, C_1$ and $D_1$ satisfy the relations in [\textbf{R}].

\noindent$\text {Proof of } \textbf{[R]}\Rightarrow \textbf{[D]}:$ Suppose $f$ is of the form \eqref{realization-general}. Then using the relations satisfied by the operators $A, B, C$ and $D$ we have
\begin{align*}
	&f(x)+f(y)^*\\
	&= A + B \pi (E(x))\left(I- D \pi (E(x)) \right)^{-1} C + A^*+ C^*\left(I-\pi(E(y))^*D^* \right)^{-1} \pi(E(y))^*B^* \\
	&=BB^* + B \left(I-\pi(E(x))D \right)^{-1} \pi(E(x)) DB^* + B D^*\pi(E(y))^* \left(I-D^*\pi(E(y))^* \right)^{-1} B^* \\
	&= \eta(x) \left[\left(I- \pi (E(x)) D \right)\left(I- D^* \pi (E(y))^* \right) \right. \\
	& \left. +   \pi(E(x))D \left(I-D^* \pi(E(y))^* \right) +   \left(I-  \pi(E(x)) D \right) D^* \pi (E(y))^*  \right] \eta(y)^* 
	\\
	&=\eta(x) \pi \left( 1-E(x) \overline{E(y)} \right) \eta(y)^*,
\end{align*}
where $\eta(x)=B \left(I- \pi (E(x))D \right)^{-1}$. Define, $\Lambda: X \times X \rightarrow \cB\left(C_b(\Psi), \cB(\cE) \right)$ by
$$
\Lambda(x,y)(g)=  \eta(x) \pi (g) \eta(y)^*.
$$
Then by Theorem \ref{KD}, $\Lambda$ is a completely positive kernel on $X$. Also,
$$
f(x)+f(y)^*= \Lambda(x,y)(1-E(x)\overline{E(y)}).
$$

\noindent$\text {Proof of } \textbf{[D]}\Rightarrow \textbf{[H]}:$ Assume that $f$ satisfies \eqref{EqnD2}. Let $$F(x)= \left(I-f(x) \right) \left(I+f(x) \right)^{-1}.$$ Then a simple computation shows that,
\begin{align*}
	I- F(x)F(y)^* &=2  \left(I+f(x) \right)^{-1} \left[f(x)+ f(y)^* \right]  \left(I + f(y)^* \right)^{-1} \\
	&=2  \left(I+f(x) \right)^{-1} \Lambda (x,y) \left( 1-\overline{E(y)}E(x)\right) \left(I + f(y)^* \right)^{-1} \\
	&=\Xi(x,y)\left( 1-\overline{E(y)}E(x)\right),
\end{align*}
where $\Xi: X \times X \rightarrow \cB \left(C_b(\Psi), \cB(\cE) \right)$ is given by $$ \Xi(x,y)(g) = 2  \left(I+f(x) \right)^{-1}\Lambda (x,y) (g) \left(I + f(y)^* \right)^{-1} \ \text{for}\ g \in C_b(\Psi).$$
Therefore by Theorem 1 in \cite{Tirtha-Ani-Vik}, we can conclude that $g$ is in $\Psi$-Schur-Agler class provided $\Xi$ is a completely positively kernel. Indeed, a similar argument used at the end of the proof of the implication \textbf{[H]} $\Rightarrow$ \textbf{[D]}, proves that $\Xi$ is a completely positive kernel.

As an application of  \cref{Main}, we have the following extension result. This gives another illustration of the power of realizations.
\begin{corollary}
	Let $Y$ be a subset of $X$ and let $f: E\rightarrow \mathcal{B}(\mathcal{E})$ such that $f$ is in the $\Psi$- Herglotz-Agler class. Then the following are equivalent:
	\begin{itemize}
		
		\item[(1)]
		there exists a function  $F: X\rightarrow \mathcal{B}(\mathcal{E})$ with $ F$ is in the $\Psi$-Herglotz-Agler class  such that
		$$F|_{Y}= f ;$$
		\item [(2)]
		there exists a completely positive kernel $\Lambda: X\times X\rightarrow \mathcal{B}\left( C_b(\Psi), \mathcal{B}(\cE)\right)$ such that
		\begin{align}\label{Eq:Ext}
			f(x)+ f(y)^* = \Lambda (x,y) \left(1-  E(x) \overline{E(y)}\right)\ \text{ for } x, y \in Y.
		\end{align}
	\end{itemize}
\end{corollary}
\begin{proof}
The implication $(1) \Rightarrow (2)$ is obvious. To establish the reverse implication, we make use of \cref{Main}.
 Using $\textbf{[H]}\Rightarrow \textbf{[D]}$ of  \cref{Main} to $f$ satisfying \eqref{Eq:Ext}, one observe that $f$ is $\Psi$-Herglotz-Agler class function on $Y$. It remains to extend $f$ to $F$ on $X$ which still satisfies $\textbf{[H]}$. To do this, we apply $\textbf{[H]}\Rightarrow \textbf{[R]}$ part of \cref{Main} to $f$ to get a $\Psi$-realization formula for $f$. Now we use this realization formula to get the extension $F$ on all of $X$. Finally, $\textbf{[R]}\Rightarrow \textbf{[H]}$ shows that $F$ is $\Psi$-Herglotz-Agler class function as required.
\end{proof}
\begin{remark}
	In general, we typically work with holomorphic test functions. In \cite{Tirtha-Ani-Vik}, it is demonstrated that holomorphic test functions inherently satisfy the condition $\psi(x_0)=0$ as stated in  \cref{Main}. Therefore,  our assumption on test functions that for all $\psi\in \Psi$, $\psi(x_0)=0$ for some  point $x_0\in X$, is not too rigid.
\end{remark}
\noindent The next remark can be viewed as a generalization of the classical Herglotz realization alluded in \eqref{RF} within this abstract framework.  
\begin{remark}\label{MRem}
It is worth mentioning that by decomposing $f(x_0) = A $ into its self-adjoint and skew-adjoint components with the help of item $ [\bold{R}] $ of \cref{Main}, and using the associated relations among $ A, B, C$, and $ D $, we can prove that a function $f$ belongs to the $ \Psi $-Herglotz-Agler class if and only if  
\begin{align}\label{MRea}
	f(x) = i \operatorname{Im} f(x_0) + V^* \left( I + U \pi(E(x)) \right) \left( I - U \pi(E(x)) \right)^{-1} V
\end{align}
for some unitary $ U $ and a bounded operator $ V $ satisfying $ V^* V = \operatorname{Re} f(x_0) $.
\end{remark}

\section{Some Concrete cases}

There are several known cases where the collection of test functions has been identified.  In the following, we will explore these instances in greater detail. 

\noindent
\textbf{\em The polydisc:} In this case the collection of test functions is just the $d-$many functions, namely the co-ordinate functions. 

\noindent
\textbf{\em Multi-connected domains:} The test functions on multi-connected domains are known \cite{DM}, and the number of such functions is uncountable. Moreover, the test functions in these domains are certain inner functions. For more details on inner functions, please refer to \cite{Fisher}. 

\noindent One prime example of a multi-connected domain is the annulus. Fix $q \in (0, 1) $ and define the annulus
$$
\mathbb{A}_q = \{z \in \mathbb{C} : q < |z| < 1\}.
$$
For normalization, fix a base point $c \in \mathbb{A}_q $ such that $ |c| \neq \sqrt{q} $. Given $\alpha \in \mathbb{T} $, there exists a holomorphic map $ \theta_\alpha $ on $\mathbb{A}_q $ that is unimodular on $ \partial \mathbb{A}_q $, has zeros at $ c $ and $ \frac{q}{\alpha c} $, and satisfies $\theta_\alpha(1) = 1 $. The set 
$$
\left\{\theta_\alpha : \alpha \in \mathbb{T}\right\}
$$
is a collection of test functions for $ H^\infty(\mathbb{A}_q) $, which is compact in the norm topology of $ H^\infty(\mathbb{A}_q) $. For more details, we refer the reader to \cite{DM}.

\noindent
\textbf{\em Quotient domains related to the bidisc:} For $d = 2$, consider the subgroup $G(m, t, 2)$, parametrized by positive integers $m$ and $t$, with $t$ dividing $m$, of the automorphism group $\operatorname{Aut}(\mathbb{D}^2)$ of $\mathbb{D}^2$:
$$
\boldsymbol{z} = (z_1, z_2) \mapsto \left( \alpha^{\nu_1} z_{\gamma(1)}, \alpha^{\nu_2} z_{\gamma(2)} \right),
$$
where $\alpha = e^{\frac{2 \pi i}{m}}$, $\gamma \in S_2$ (the permutation group of 2 symbols), and $\nu_1, \nu_2$ are integers whose sum is divisible by $t$. The group $G(m, t, 2)$ is a pseudo-reflection group of order $\frac{m^2 \cdot 2!}{t}$, and it appears in the classification of pseudo-reflection groups. We consider the quotient domain $\mathbb{D}^2 / G(m, t, 2)$, which is biholomorphic to the image $\theta(\mathbb{D}^2)$ of $\mathbb{D}^2$ under the proper holomorphic map $\theta = (\theta_1, \theta_2): \mathbb{D}^2 \to \theta(\mathbb{D}^2)$ given by
\[
\theta_1(\boldsymbol{z}) = E_1(z_1^m, z_2^m) \text{ and } \theta_2(\boldsymbol{z}) = \left[ E_2(z_1^m, z_2^m) \right]^{\frac{1}{t}} = (z_1 z_2)^{\frac{m}{t}},
\]
where $E_1$ and $E_2$ are the elementary symmetric polynomials in two variables; see \cite{Rudin-IUMJ}. Generally, these quotient domains are non-convex, not star-like, and do not have a $C^2$-smooth boundary. One of the primary examples of a quotient domain is the symmetrized bidisc, $\bG := \mathbb{D}^2 / G(1, 1, 2)$. The symmetrized bidisc has been the focus of substantial developments in both function theory and operator theory.

\noindent
The set of test functions on $\mathbb{G}$ is known and represented by 
$
\{\varphi_\alpha: \alpha \in \mathbb{T}\}$ 
 where  
$$\varphi_\alpha(s, p) = \frac{2\alpha p - s}{2 - \alpha s} \quad \text { for } (s, p) \in \mathbb{G}.$$
These functions appear as solutions to the two-point Nevanlinna-Pick interpolation problem, as discussed in \cite{KZ}; see also \cite{DKS, DKS2}. Recently, it was shown in \cite{BK2} that the test functions for
$\theta(\mathbb{D}^2)$ 
are given by $$\psi_\alpha(p_1, p_2) = \frac{2\alpha p_2^t - p_1}{2 - \alpha p_1} \quad (p_1, p_2) \in \theta(\mathbb{D}^2). $$

\noindent
\textbf{\em Distinguished varieties:}  Distinguished varieties are certain algebraic variety that plays a key role in function theory and operator theory, particularly in the study of the Pick-Nevanlinna interpolation problems and improvement of Andô's inequality. A distinguished variety lies in the polydisc $\mathbb{D}^d$ and intersects the boundary of the polydisc \(\partial(\mathbb{D}^d)\) only in a subset of the d-torus $\mathbb{T}^d = \{(z_1, \cdots, z_d) \in \mathbb{C}^d : |z_1| = 1, \cdots, |z_d| = 1\}$.  An example of a distinguished variety is the Neil parabola:
$$\cN=\{(z_1, z_2): z_1^2=z_2^3\}.$$
For further details on distinguished varieties, see \cite{AM, BKS}. 
Test functions are also known for certain distinguished varieties (see \cite{DM, DU}).

\noindent In these examples, the test functions are holomorphic, with a point in the domain where all test functions vanish. The classical motivation for the results in this paper stems from the polydisc. As for the symmetrized bidisc and the annulus, we will provide further elaboration below. Having several realization formulae for holomorphic functions on the symmetrized bidisc \cite{SYM_Real, Tirtha-Hari-JFA} and the annulus \cite{DM}, with values in the norm unit ball of $\mathcal{B}(\mathcal{E}$), leads to additional equivalent statements. We have the following realization formula for Herglotz class functions in these domains.
\begin{thm}
	Let $\Omega=\mathbb{G}\text{ or } \mathbb{A}_q$ and $f:\Omega\rightarrow \mathcal{B}(\mathcal{E})$ be a function. Let $\{\psi_{\alpha}: \alpha\in \bT\}$ denote the collection of test functions for $\Omega$ as defined above. We write, $\psi_\alpha(w) = \psi(\alpha, w)$ for $w\in \Omega$ and $\alpha \in \bT$. Then the following are equivalent:
	
	\begin{itemize}
		\item[\textbf{[H]}]
		$f$ is in $\mathfrak H(\Omega, \mathcal B (\mathcal E))$;
		
		\item [\textbf{[D1]}]
		there exist a Hilbert space $\mathcal{K}$, a holomorphic function $h:\Omega\rightarrow\mathcal{B}(\mathcal{K}, \mathcal{E})$ and a unitary operator $\tau$ on $\mathcal{K}$ such that
		
		\begin{align}\label{D1}
			f(z)+ f(w)^* = h (z)\left(I_{\mathcal{K}}- \psi(\tau, z)\psi(\tau, w)^*\right) h(w)^*;
		\end{align}
		
		\item [\textbf{[D2]}]
		there exists a completely positive kernel $\Lambda: \Omega\times \Omega\rightarrow \mathcal{B}\left( C(\mathbb{T}), \mathcal{B}(\cE)\right)$ such that
		\begin{align}\label{D2}
			f(z)+ f(w)^* = \Lambda\left(z, w\right)\left(1-  \psi(\cdot, z)\overline{\psi(\cdot, w)}\right) ;
		\end{align}

		\item [\textbf{[R1]}]
		there exist a Hilbert space $\mathcal{K}$, a unitary operator $\tau$ on $\mathcal{K}$ and a bounded colligation operator matrix
		\begin{align*}
			U= \begin{bmatrix}
				A & B \\
				C & D
			\end{bmatrix}: \mathcal{E}\oplus\mathcal{K}\rightarrow \mathcal{E}\oplus\mathcal{K}
		\end{align*} such that
		\begin{align*}\label{EqR1}
			f(z) = A + B\psi ( \tau, z) \left(I- D\psi (\tau, z)\right)^{-1}C
		\end{align*}
		where $D^*D= DD^*= I_\mathcal{K}$, $D^*C= B^*$ and $C^*C=BB^*= A+A^*$ and
		
		\item [\textbf{[R2]}]
		there exist a Hilbert space $\mathcal{K}$, a unital $*$-representation $\rho: C(\mathbb{T})\rightarrow\mathcal{B}(\mathcal{K})$ and a bounded colligation operator matrix
		$$U= \begin{bmatrix}
			A & B \\
			C & D
		\end{bmatrix}: \mathcal{E}\oplus\mathcal{K}\rightarrow \mathcal{E}\oplus\mathcal{K}$$
		such that 	\begin{align*}
			f(z) = A + B\rho\left(\psi(\cdot, z)\right) \left(I- D\rho\left(\psi (\cdot , z)\right)\right)^{-1}C
		\end{align*}
		where $D^*D= DD^*= I_\mathcal{K}$, $D^*C= B^*$ and $C^*C=BB^*= A+A^*$.
	\end{itemize}
	
\end{thm}
\begin{proof}[\textbf{Sketch of the proof:}]
	The proof mainly relies on  \cref{Main}. Since the collection of test functions is parametrized by the circle, so $C_b(\Psi)$ is nothing but the uniform algebra $C(\bT)$.
	First note that the statements \textbf{[H]}, \textbf{[D2]} and \textbf{[R2]} are integral components of  \cref{Main}. We shall demonstrate the following$$\textbf{[D1]}\Leftrightarrow \textbf{[D2]} \text{ and } \textbf{[R1]}\Leftrightarrow \textbf{[R2]}.$$
	
\noindent	$\textbf{[D1]}\Leftrightarrow \textbf{[D2]}:$ This equivalence mainly depends on the Kolmogorov decomposition (Theorem \ref{KD}). Let $\cC(\tau)$ be the unital $\cC^*$-algebra generated by $\tau$ and $\rho:C(\sigma(\tau))\rightarrow \cC(\tau)$ be the inverse of the Gelfand transform. Therefore, $$\rho\left(\psi(\cdot, z)|_{\sigma(\tau)}\right)= \psi(\tau, z).$$ Now consider the completely positive kernel  $\Lambda: \Omega\times \Omega\rightarrow \mathcal{B}\left( C(\mathbb{T}), \mathcal{B}(\cE)\right)$ given by
	$$\Lambda(z, w)(g)= h(z)\rho(g|_{\sigma(\tau)}) h(w)^*$$
	for all $g\in C(\bT)$. Therefore, equation \eqref{D2} can be derived from equation \eqref{D1}. For the converse part, apply Kolmogorov decomposition (Theorem  \ref{KD}) on the kernel $\Lambda$ to get \eqref{D1}.
	
\noindent	$\textbf{[R1]}\Leftrightarrow \textbf{[R2]}:$  This equivalence  follows by noting that representation can be viewed as a unitary operator. Indeed, if we have a unitary operator $\tau$, then consider the represention $\rho$ on $C(\sigma(\tau))$ given by $$\rho\left(\psi(\cdot, z)|_{\sigma(\tau)}\right)= \psi(\tau, z)$$ where $\psi$ is as above. Conversely, if we have a represention $\rho: C(\bT)\rightarrow \mathcal{B}(\cK)$, then consider the unitary operator $\tau= \rho(I_{C(\bT)})$ where $I_{C(\bT)}$ is the identity function on the unit circle.
\end{proof}
We shall end this article with the following remark. 
\begin{remark}
If we aim to find an integral representation, as in equation \eqref{MRea}, for functions in $ \mathfrak{H}(\bG, \cB(\cE))$ that map $(0,0)$ to the identity operator on \( \cE \), we can use lurking isometry arguments instead of Krein space theory. Specifically, any Herglotz class function $f$ on the symmetrized bidisc with $f(0, 0) = I_{\cE} $ can be expressed as:
\begin{align*}
	f(s, p) = V^*\left(I + U \varphi(\tau, (s, p))\right) \left(I - U \varphi(\tau, (s, p))\right)^{-1} V, \quad (s, p) \in \bG,
\end{align*}
for some Hilbert space $ \mathcal{K} $, unitary operators $ \tau$, $U $ on $\mathcal{E} $, an isometry $V $ from $ \cE $ to $\cK$, and $ \varphi(\tau, (s, p)) = (2p \tau - s)(2 - s \tau)^{-1} $. We leave the verification of this result to the reader.

\end{remark}

\textbf{Funding:} The first author's work is supported by the Prime Minister's Research Fellowship PMRF-21-1274.03 and the second author is partially supported by a PIMS postdoctoral fellowship. 

The authors thank the anonymous referee for their comments and valuable suggestions.

\end{document}